\documentclass[11pt,a4paper]{amsart}
\usepackage[utf8]{inputenc}
\usepackage[english]{babel}
\usepackage[left=1.5cm,right=1.5cm,top=2cm,bottom=2cm,foot=1cm]{geometry}
\usepackage{amsmath,amssymb,latexsym, amsthm}
\usepackage{hyperref}
\renewcommand{\d}{\mathrm{d}}		

\newcommand{\ev}{\mathrm{ev}}

\newcommand{\vol}{\mathrm{vol}}

\newcommand{\relint}{\mathrm{relint\,}}

\newcommand{\Vol}{\mathrm{vol}}

\newtheorem{thm}{Theorem}

\newtheorem{cor}{Corollary}
\newtheorem{prop}{Proposition}

\theoremstyle{definition}
\newtheorem*{rem}{Remark}
\newtheorem{defin}{Definition}

\begin{document}
\title{On the Volume of Boolean expressions of Large Congruent Balls}
\author{Bal\'azs Csik\'os}
\address{E\"otv\"os Lor\'and University, Institute of Mathematics}
\dedicatory{Dedicated to K\'aroly Bezdek and Egon Schulte on the occasion of their 60th birthdays.}
\curraddr{Budapest, P\'azm\'any P\'eter stny. 1/C, H-1117 Hungary}
\email{csikos@cs.elte.hu}
\date{}
\begin{abstract}  We consider  the volume of a Boolean expression of some congruent
balls about a given system of centers in the $d$-dimensional Euclidean space. When the radius $r$ of the balls is large, this volume can be approximated by a polynomial of $r$, which will be computed up to an $O(r^{d-3})$ error term. We study how the top coefficients of this polynomial depend on the set of the centers. It is known that in 
the case of the union of the balls, the top coefficients are some constant multiples 
of the intrinsic volumes of the convex hull of the centers. 
Thus, the coefficients in the general case lead to generalizations of the 
intrinsic volumes, in particular, to a generalization of the mean width of a 
set. Some known results on the mean width, along with the theorem on its 
monotonicity under contractions are extended to the ``Boolean analogues'' of the 
mean width.
\end{abstract}

\maketitle

\section{Introduction}\label{sec:introduction}
The long-standing conjecture of Kneser \cite{Kneser} and Poulsen 
\cite{Poulsen} claims that if the points $\mathbf p_1,\dots,\mathbf p_N$ and 
$\mathbf q_1,\dots,\mathbf q_N$ of the $d$-dimensional Euclidean space $\mathbb 
R^d$ satisfy the inequalities $\d(\mathbf p_i,\mathbf p_j)\geq \d(\mathbf 
q_i,\mathbf q_j)$ for all $0\leq i,j\leq N$, then 
\[
\vol_d\left(\bigcup_{i=1}^N B^d(\mathbf p_i,r)\right)\geq \vol_d\left(\bigcup_{i=1}^N B^d(\mathbf q_i,r)\right)
\]
for any $r>0$, where $B^d(\mathbf p,r)$ denotes the closed $d$-dimensional ball 
of radius $r$ about the point $\mathbf p$ and $\vol_d$ is the $d$-dimensional 
volume. K.~Bezdek and Connelly \cite{Bezdek_Connelly_2002} proved the 
conjecture in the plane, but it is still open in dimensions $d\geq 3$. 

Results of Gromov \cite{Gromov}, Gordon and Meyer 
\cite{Gordon_Meyer} and the author \cite{Csikos_flower}, \cite{Csikos_Schlafli} 
suggest that the Kneser--Poulsen conjecture could be true in a more general 
form, which we formulate below. 

Let $\mathcal B_N$ be the free Boolean algebra generated by $N\geq 1$ symbols 
$x_1,\dots,x_N$.  We denote the greatest  element of $\mathcal B_N$ by $X$, and the 
least element of $\mathcal B_N$ by $\emptyset$. Elements of $\mathcal B_N$ are 
equivalence classes of formal expressions built from the symbols 
$x_1,\dots,x_N$, $X$ and $\emptyset$, the binary operations $\cup$, $\cap$, 
and the unary operator $f\mapsto \bar{f}$. Two expressions are called equivalent 
if and only if we can prove their equality assuming that the operations satisfy 
the axioms of a Boolean algebra. We shall refer to an element of $\mathcal B_N$ 
by choosing a Boolean expression from its equivalence class, and we write 
``$=$'' between two Boolean expressions if they are equivalent. We shall also 
use the derived operator $f\setminus g=f\cap \bar{g}$ and the partial ordering 
$f\subseteq g\overset{\mathrm{def}}{\iff} f\cup g=g$.  We refer to 
\cite{Givant_Halmos} for more details on Boolean algebras.

Take a Boolean expression $f\in \mathcal B_N$ which can be represented by a 
formula built exclusively from the variables $x_1,\dots,x_N$ and the operations 
$\cup$, $\cap$, $\setminus$ in such a way that each of the variables occurs in 
the formula exactly once. For any pair of indices $i\neq j$, $1\leq i,j\leq N$, 
evaluate $f$ replacing the variables $x_k$, $k\notin\{i,j\}$  by $X$ or 
$\emptyset$ in all possible ways. It can be seen that the results of those 
evaluations that are not equal to   $X$ or $\emptyset$, are all equal to one another and to one of the expressions $x_i\cap x_j$, $x_i\setminus x_j$, $x_j\setminus 
x_j$, $x_i\cup x_j$. Let the sign $\epsilon_{ij}^f$ be $-1$ if the evaluations 
not equal to $X$ or $\emptyset$ are equal to $x_i\cap x_j$, and set 
$\epsilon_{ij}^f=1$ in the remaining three cases.

The generalization of the Kneser--Poulsen conjecture for Boolean 
expressions of balls claims that if the Boolean expression $f\in \mathcal B_N$ obeys  the conditions of the previous paragraph, and the points $\mathbf 
p_1,\dots,\mathbf p_N$ and $\mathbf q_1,\dots,\mathbf q_N$ in $\mathbb R^d$ 
satisfy the inequalities $\epsilon_{ij}^f(\d(\mathbf p_i,\mathbf p_j)- 
\d(\mathbf q_i,\mathbf q_j))\geq 0$ for all $0\leq i,j\leq N$, then 
\begin{equation}\label{general_Kneser}
\vol_d\left(f( B^d(\mathbf p_1,r_1),\dots,B^d(\mathbf p_N,r_N))\right)\geq 
\vol_d\left(f( B^d(\mathbf q_1,r_1),\dots,B^d(\mathbf q_N,r_N))\right)
\end{equation}
for any choice of the radii $r_1,\dots,r_N$. 

A suitable modification of the arguments of Bezdek and Connelly 
\cite{Bezdek_Connelly_2002} shows that this generalization of the 
Kneser--Poulsen conjecture is also true in the Euclidean plane (see 
\cite{Csikos_Schlafli}).

As it was pointed out by Capoyleas, Pach  \cite{Capoyleas_Pach}, and 
Gorbovickis \cite{Gorbovickis_strict}, the original Kneser--Poulsen 
conjecture for large congruent balls is closely related to the monotonicity of 
the mean width of a set under contractions. The relation is based on the formula
\begin{equation}\label{union_asymptotics}
\vol_d\left(\bigcup_{i=1}^N B^d(\mathbf p_i,r)\right)=\kappa_d 
r^d+\frac{d\kappa_d}{2}\boldsymbol{\omega}_d(\{\mathbf p_1,\dots,\mathbf p_N\})r^{d-1}+O(r^{d-2}),
\end{equation}
where $\kappa_d$ is the volume of the unit ball in $\mathbb R^d$, $\boldsymbol{\omega}_d(S)$ denotes the mean width of the bounded set $S\subset \mathbb R^d$. We remark that the mean width function $\boldsymbol{\omega}_d$ depends on the dimension $d$ of the ambient space, but only up to a constant factor. More explicitly, if $\Phi\colon\mathbb R^d\to \mathbb R^{\tilde d}$ is an isometric embedding, then we have $\frac{d\kappa_d}{\kappa_{d-1}}\boldsymbol{\omega}_d(S)=\frac{{\tilde d}\kappa_{\tilde d}}{\kappa_{{\tilde d}-1}}\boldsymbol{\omega}_{\tilde d}(\Phi(S))$ for any bounded set $S\subset \mathbb R^d$. Applying formula \eqref{union_asymptotics} and the fact that the Kneser--Poulsen conjecture is true if the dimension of the space is at least $N-1$ (see \cite{Gromov}),  Capoyleas and Pach  \cite{Capoyleas_Pach} proved that the mean width of a set cannot increase when the set is contracted. Using rigidity theory, Gorbovickis  \cite{Gorbovickis_strict} sharpened this result by proving that if the $d$-dimensional configurations $(\mathbf p_1,\dots,\mathbf p_N)$ and 
$(\mathbf q_1,\dots,\mathbf q_N)$ are not congruent and satisfy the inequalities $\d(\mathbf p_i,\mathbf p_j)\geq \d(\mathbf 
q_i,\mathbf q_j)$ for all $0\leq i,j\leq N$, then the strict inequality 
\[
\boldsymbol{\omega}_d(\{\mathbf p_1,\dots,\mathbf p_N\})>\boldsymbol{\omega}_d(\{\mathbf q_1,\dots,\mathbf q_N\})
\]

holds. This strict inequality, in return,  implies that the Kneser--Poulsen conjecture is true if the radius of the balls is bigger than a constant depending on the configurations of the centers.

Gorbovickis  \cite{Gorbovickis_strict}  proved also that for the volume of the intersection of large congruent balls we have 
\begin{equation}\label{intersection_asymptotics}
\vol_d\left(\bigcap_{i=1}^N B^d(\mathbf p_i,r)\right)=\kappa_d 
r^d-\frac{d\kappa_d}{2}\boldsymbol{\omega}_d(\{\mathbf p_1,\dots,\mathbf p_N\})r^{d-1}+O(r^{d-2}),
\end{equation}
thus, as a consequence of the strict monotonicity of the mean width, the above mentioned generalization of the Kneser--Poulsen conjecture is true also for the intersections of congruent balls if the radius of the balls is greater than a constant depending on the configurations of the centers.

In 2013 K.~Bezdek \cite{Bezdek_Fields} posed the problem of finding a suitable generalization of equations \eqref{union_asymptotics} and \eqref{intersection_asymptotics} for the volume of an arbitrary  Boolean expression of large congruent balls, and suggested to explore the interplay between the generalized Kneser--Poulsen conjecture and the monotonicity properties of the coefficient of $r^{d-1}$ in the general formula. In the present paper, we summarize the results of the research initiated by these questions.

The outline of the paper is the following. In Section \ref{sec:2}, we sharpen equation \eqref{union_asymptotics}, expressing the volume of the union of some large congruent balls with an error term of order $O(r^{d-3})$. The coefficients appearing in the formula are some constant multiples of the intrinsic volumes $V_0$, $V_1$, $V_2$ of the convex hull of the centers. In Section \ref{sec:3}, we show that if a Boolean expression $f(B_1,\dots, B_n)$ of some balls is bounded, then its volume can be obtained as a linear combination of the volumes of the unions of some of the balls. The coefficients of this inclusion-exclusion type formula, given in  Proposition \ref{prop:decomp}, depend purely on the Boolean expression $f$. These coefficients are used to define the Boolean analogues of the intrinsic volumes of the convex hull of a point set in Section \ref{sec:4}. Theorem \ref{flower_volume} gives a generalization of equation \eqref{union_asymptotics} for Boolean expressions of large balls using Boolean intrinsic volumes. In Section \ref{sec:5}, some classical facts on intrinsic volumes are generalized for Boolean intrinsic volumes. For example, it is known that the $k$th intrinsic volume of a polytope can be expressed in terms of the volumes of the $k$-dimensional faces and the angular measures of the normal cones of these faces. This formula is generalized for Boolean intrinsic volumes in Theorem \ref{thm:2}.  As an application of Theorem \ref{thm:2}, we prove that  the   $k$th Boolean intrinsic volumes corresponding to dual Boolean expressions differ only in a sign $(-1)^k$. This explains why the coefficients of $r^{d-1}$ in the equations \eqref{union_asymptotics} and \eqref{intersection_asymptotics} are opposite to one another. Theorem \ref{thm:3} provides a Boolean extension of the fact that the first intrinsic volume of a convex set is a constant  multiple of the integral of its support function. Section \ref{sec:6} is devoted to the proof of Theorem \ref{T:V_f,1}  on the monotonicity of the Boolean analogue of the first intrinsic volume.

\section{Comparison of the volume of a union of balls and the volume of its convex hull \label{sec:2}}
Every convex polytope $K\subset \mathbb R^d$ defines a decomposition of the 
space as follows.  Denote by $\mathcal F(K)$ the set of all faces of $K$, 
including $K$, and by $\mathcal F_k(K)$ the set of its $k$-dimensional 
faces. Let $\pi\colon \mathbb R^d\to K$ be the map assigning to a point $\mathbf 
x\in \mathbb R^d$ the unique point of $K$ that is closest to $\mathbf x$. For a 
face $L\in \mathcal F(K)$,  denote by $V(L,K)$ the preimage $\pi^{-1}(\relint 
L)$ of the relative interior of $L$.
As $K$ is the disjoint union of the relative interiors of its faces, $\mathbb R^d$ is the disjoint union of the sets $V(L,K)$, where $L$ is running over $\mathcal F(K)$.
If $L\in\mathcal F_k(K)$, then $V(L,K)$ is the Minkowski sum of the relative interior of $L$ and the normal cone 
\begin{equation}
N(L,K)=\{\mathbf u\in \mathbb R^{d}\mid \mathbf u\perp [L]\text{ and } \max_{\mathbf x\in K}\langle \mathbf u,\mathbf x\rangle \text{ is attained at a point }\mathbf x\in L\}
\end{equation}
of $K$ at $L$, where $[L]$ denotes the affine subspace spanned by $L$. Set  $n(L,K)=N(L,K)\cap B^d(\mathbf 0,1)$ and $\nu(L,K)=\Vol_{d-k}(n(L,K))/\kappa_{d-k}$. Division by $\kappa_{d-k}$ in the definition of $\nu(L,K)$ is advantageous because it makes the angle measure $\nu(L,K)$ of the normal cone $N(L,K)$ independent of the dimension $d$ of the ambient space $\mathbb R^d$, though the normal cone itself changes if we embed $K$ into a higher dimensional space. 

Denote by  $K_r=K+B^d(\mathbf 0,r)$ the distance $r$ parallel body of $K$. The decomposition 
\begin{equation}\label{decomposition}
\mathbb R^d=\bigcup_{L\in \mathcal F(K)}N(L,K)
\end{equation}
induces a decomposition of the parallel body $K_r$, which enables us to write the volume of $K_r$ as a polynomial of $r$
\begin{equation}\label{Steiner1}
\begin{aligned}
\Vol_d(K_r)&=\sum_{L\in \mathcal F(K)}\Vol_d(K_r\cap N(L,K))=\sum_{L\in \mathcal F(K)}\Vol_d(L+ r\,(n(L,K))\\&=\sum_{k=0}^d \kappa_{d-k}\left(\sum_{L\in \mathcal F_k(K)}\Vol_k(L) \nu(L,K)\right)r^{d-k}.
\end{aligned}
\end{equation}
Equation \eqref{Steiner1} is a special case of Steiner's classical formula (see, e.g., \cite[equation (4.2.27)]{Schneider})
\begin{equation} \label{Steiner2}
\Vol_d(K+B(\mathbf 0,r))= \sum_{k=0}^d\binom{d}{k}W_k^d(K)r^k=\sum_{k=0}^d \kappa_{d-k}V_k(K)r^{d-k},
\end{equation}
expressing the volume of the distance $r$ parallel body of an arbitrary compact convex set $K$ as a polynomial of $r$, in which the normalized coefficients $W_k^d(K)$ and $V_k(K)$ are the quermassintegrals and intrinsic volumes of $K$ respectively. It is known that the intrinsic volumes are continuous functions on the space of compact convex sets endowed with the Hausdorff metric (see \cite[Section 4.2]{Schneider}), and $V_0(K)\equiv 1$. Comparing \eqref{Steiner1} and \eqref{Steiner2} we obtain the formula 
\begin{equation}\label{intrinsic_volume}
V_k(K)=\sum_{L\in \mathcal F_k(K)}\Vol_k(L) \nu(L,K)
\end{equation}
expressing the intrinsic volumes of a polytope $K$.
\begin{prop} \label{prop:comparison}
Let $\mathbf p_1,\dots,\mathbf p_N$ be a fixed set of points in $\mathbb R^d$, $K=\mathrm{conv}(\{\mathbf p_1,\dots,\mathbf p_N\})$ be the convex hull of the points. Denote by $B_i=B^d(\mathbf p_i,r)$ the ball of radius $r$ centered at $\mathbf p_i$. Then we have
\begin{equation} 
\left|\Vol_d(K_r)-\Vol_d\Big(\bigcup_{i=1}^N B_i\Big)\right|=O(r^{d-3})
\end{equation}
for large values of $r$.
\end{prop}
\begin{proof} Denote by $\Delta$ the diameter of $K$, and set $r'=r-\Delta^2/r$. It is easy to see that if $r\geq \Delta$, then $K_{r'}\subseteq \bigcup_{i=1}^NB_i\subseteq K_r$, (see \cite{Capoyleas_Pach}). Intersecting the decomposition \eqref{decomposition} with the union of the balls, we get
\[
\bigcup_{i=1}^NB_i=\bigcup_{L\in\mathcal F(K)}\Big(N(L,K)\cap\Big(\bigcup_{i=1}^NB_i\Big)\Big).
\]
When $L\in\mathcal F_0(K)$ is a vertex, we have $N(L,K)\cap\left(\bigcup_{i=1}^NB_i\right)=N(L,K)\cap K_r$. Thus,
\begin{equation*} 
K_r\setminus \Big(\bigcup_{i=1}^N B_i\Big)\subseteq \bigcup_{k=1}^d
\bigcup_{L\in\mathcal F_k(K)} N(L,K)\cap(K_r\setminus K_{r'}),
\end{equation*}
and 
\begin{equation} 
\left|\Vol_d(K_r)-\Vol_d\Big(\bigcup_{i=1}^N B_i\Big)\right|\leq \sum_{k=1}^d \Vol_k(L)\kappa_{d-k}\nu(L,K)\left(r^{d-k}-\left(r-\frac{\Delta^2}{r}\right)^{d-k}\right)=O(r^{d-3}),
\end{equation}
as claimed.
\end{proof}
\begin{cor}\label{Cor:union_asymptotics} Using the notations of Proposition \ref{prop:comparison}, we have
\begin{equation} \label{eq:union}
\Vol_d\Big(\bigcup_{i=1}^N B_i\Big)=\kappa_dr^d+\kappa_{d-1}V_1(K)r^{d-1}+\kappa_{d-2}V_2(K)r^{d-2}+O(r^{d-3}).
\end{equation} 
\end{cor}

\section{Combinatorics of Boolean expressions \label{sec:3}}

For a subset $I$ of the set $[N]=\{1,\dots,N\}$, define $a_I\in \mathcal B_N$ by $a_I=(\bigcap_{j\notin I}x_j) \setminus (\bigcup_{i\in I}x_i)$. The elements $a_I$, $(I\subseteq [N])$ are the \emph{atomic elements} of $\mathcal B_N$. Any $f\in \mathcal B_N$ can be decomposed uniquely as $f=\bigcup_{a_I\subseteq f}a_I$. In particular, $\mathcal B_N$ has $2^{2^N}$ elements.

\begin{defin} The \emph{reduced Euler characteristic  $\tilde\chi_N(f)$ of $f\in \mathcal B_N$} is the integer $\tilde\chi_N(f)=\sum_{a_I\subseteq f}(-1)^{|I|+1}$.
\end{defin}
Obviously, the reduced Euler characteristic of a Boolean expression is an integer number in the interval $[-2^{N-1},2^{N-1}]$.

\begin{prop}\label{zero_Euler} If $f\in\mathcal B_N$  can be represented by a formal expression which does  not contain all the variables $x_1,\dots,x_N$, then  $\tilde\chi_N(f)=0$.

\end{prop}
\begin{proof}We may assume without loss of generality that $f$ can be written as an expression not using the variable $x_N$. This means that if $\iota\colon \mathcal B_{N-1}\to \mathcal B_{N}$ is the natural embedding, then $f=\iota(g)$ for some $g\in \mathcal B_{N-1}$. If $I\subseteq [N-1]$, and $a_I\in \mathcal B_{N-1}$ is the corresponding atomic expression in $\mathcal B_{N-1}$,  then $\iota(a_I)\cap x_{N}$ and $\iota(a_I)\cap \bar x_{N}$ are atomic expressions in $\mathcal B_{N}$ corresponding to the index sets $I\subseteq [N]$ and $I\cup\{N\}\subseteq [N]$  respectively, furthermore,
\[
a_I\subseteq g \iff \iota(a_I)\cap x_{N}\subseteq f \iff \iota(a_I)\cap \bar x_{N}\subseteq f.
\] 
Thus,
\[\tilde\chi_{N}(f)=\sum_{a_I\subseteq g} ((-1)^{|I|}+(-1)^{|I\cup \{N\}|})=0.
\qedhere
\]
\end{proof}

\begin{prop}\label{complement_Euler}
If $\bar f$ is the complement of $f\in\mathcal B_N$, then $\tilde\chi_N(\bar f)=-\tilde\chi_N(f)$.
\end{prop}
\begin{proof} It is clear from the definition of the reduced Euler characteristic that if $f\cap g=\emptyset$, then $\tilde\chi_N(f\cup g)=\tilde\chi_N(f)+\tilde\chi_N(g)$. We also have 
\[\tilde\chi_N(X)=\sum_{i=0}^N(-1)^{i+1}\binom{N}{i}=0,\]
so $\tilde\chi_N(f)+\tilde\chi_N(\bar f)=\tilde\chi_N(f\cup \bar f)=\tilde\chi_N(X)=0$.
\end{proof}
Recall that the contradual $f^{\bar{*}}$ of $f\in \mathcal B_N$ is formed by replacing each variable $x_i$ by its complement $\bar x_i$, while the dual $ f^*=\overline{f^{\bar{*}}}$ of $f$ is the complement of the contradual of $f$.
\begin{prop}\label{Euler_dual}
For any $f\in\mathcal B_N$, we have 
\[
-\tilde\chi_N(f^*)=\tilde\chi_N(f^{\bar *})=(-1)^N\tilde\chi_N(f).
\]
\end{prop}
\begin{proof}
The first equation is a corollary of Proposition \ref{complement_Euler}, so it is enough to show the second one. The contradual operation preserves the ordering and maps the atom $a_I$ to $a_{[N]\setminus I}$. Consequently,
\[
\tilde\chi_N(f^{\bar *})=\sum_{a_I\subseteq f^{\bar *} }(-1)^{|I|+1}=(-1)^N\sum_{a_{[N]\setminus I}\subseteq f} (-1)^{|[N]\setminus I|+1}=(-1)^N\tilde\chi_N(f).\qedhere
\]
\end{proof}
Let $\mathcal L_N$ be the sublattice of $\mathcal B_N$ generated by the elements $x_1,\dots,x_N$ and the operations $\cup$ and $\cap$. An element $f\in \mathcal B_n$ belongs to $\mathcal L_n$ if and only if $f\neq \emptyset$ and whenever $a_I\subseteq f$ and $J\subseteq I$ we also have $a_J\subseteq f$. This means that we can associate to any element $f\in \mathcal L_N$ an abstract  simplicial complex $P_f=\{I\subset [N]\mid a_{I}\subseteq f\}$. This assignment gives a bijection between $\mathcal L_n$ and abstract simplicial complexes on the vertex set $[N]$ different from the abstract $(N-1)$-dimensional simplex. In this special case, the reduced Euler characteristic of $f$ is one less than the ordinary Euler characteristic of $P_f$. The difference is due to the fact that $\emptyset$ is not counted as a $-1$-dimensional face when we compute the Euler characteristic, but it is taken into account in the computation of $\tilde\chi_N(f)$. The number of elements of $\mathcal L_N$ is $M_N-2$, where $M_N$ is the 
$N$th Dedekind number.

There is a sublattice $\mathcal C_N\supset \mathcal L_N$   of $\mathcal B_N$ consisting of expressions that can be built from the variables $x_1,\dots,x_N$ using only the operations $\cup$, $\cap$, and $\setminus$. The lattice $\mathcal C_N$ contains exactly those elements of $\mathcal B_N$ that do not contain the atomic expression $a_{[N]}$. This way, $\mathcal C_N$ has $2^{2^N-1}$ elements. 

Denote by $\mathcal M_N$ the linear space of real valued functions $\mu\colon \mathcal C_N\to \mathbb R$ such that $\mu(f\cup g)=\mu(f)+\mu(g)$ if $f\cap g=\emptyset$. As $\mu\in \mathcal C_N$ is uniquely determined by its values on the atomic expressions $a_I$, $(I\subsetneq [N])$, $\dim \mathcal M_N=2^N-1$.

For $\emptyset\neq I \subseteq [N]$, let $u_I\in \mathcal L_N$ be the union $u_I=\bigcup_{i\in I}x_i$.

\begin{prop}\label{prop:decomp} For any $f\in \mathcal C_N$, there is a unique collection of integers $m_{f,I}\in \mathbb Z$ for $(\emptyset\neq I\subseteq [N])$ such that for any $\mu\in \mathcal M_N$, we have 
\begin{equation}\label{decomp}
\mu(f)=\sum_{\emptyset\neq I\subseteq [N]} m_{f,I}\mu(u_I).
\end{equation}
\end{prop}
\begin{proof} There is a natural embedding $\ev\colon \mathcal C_N\to \mathcal M_N^*$ of $\mathcal C_N$ into the dual space of $\mathcal M_N$ given by the evaluation map $\ev\colon f\mapsto \ev_f$, where $\ev_f(\mu)=\mu(f)$ for any $\mu\in \mathcal M_N$. The proposition claims that for any $f\in\mathcal C_N$, $\ev_f$ can be decomposed uniquely as an integer coefficient linear combination of the evaluations $\ev_{u_I}$, $(\emptyset\neq I\subseteq [N])$.

Any $f\in \mathcal C_N$ has an atomic decomposition $f=\bigcup_{a_I\subseteq f}a_I$, showing that 
\begin{equation}\label{uniora_bontas1}
\ev_f= \sum_{a_I\subseteq f}\ev_{a_I}.
\end{equation} 
Applying the inclusion--exclusion formula 
\[
\mu\left(\bigcap_{k\in K} A_k\right)=\sum_{\emptyset\neq J\subseteq K}(-1)^{|J|+1}\mu\left(\bigcup_{j\in J}A_j\right)
\]
for the Boolean expressions $A_k=x_k\setminus u_I$, $k\in K=[N]\setminus I$,  we obtain

\begin{equation}\label{uniora_bontas2}
\begin{aligned}
\mu(a_I)&=\sum_{\emptyset\neq J\subseteq ([N]\setminus I)}(-1)^{|J|+1}\mu(u_J \setminus u_I)=\sum_{\emptyset\neq J\subseteq ([N]\setminus I)}(-1)^{|J|+1}(\mu(u_{I\cup J})-\mu(u_I))\\&=\sum_{I\subseteq K\subseteq [N]}(-1)^{|K\setminus I|+1}\mu(u_{K}),
\end{aligned}
\end{equation}
for any $\mu\in \mathcal M_N$ and $I\neq [N]$.

Equations \eqref{uniora_bontas1} and \eqref{uniora_bontas2} show that $\ev_f$ can be written as a linear combination of the evaluations $\ev_{u_I}$, $(\emptyset\neq I\subseteq [N]))$ with integer coefficients.

To show uniqueness of the coefficients $m_{f,I}$, observe that the evaluations $\ev_{a_I}$, $(\emptyset\neq I\subseteq [N])$ form a basis of $\mathcal M_N^*$, and as the linear space spanned by the $2^N-1$ evaluations $\ev_{u_I}$, $(\emptyset\neq I\subseteq [N])$ contains this basis, it is the whole space $\mathcal M_N^*$. Since $\dim \mathcal M_N^*=2^N-1$, the evaluations 
$\ev_{u_I}$, $(\emptyset\neq I\subseteq [N])$ are linearly independent. \end{proof} 

\begin{prop} The sum $\sum_{\emptyset\neq I\subseteq [N]}m_{f,I}$ of the coefficients is $1$ if $a_{\emptyset}\subseteq f$ and $0$ otherwise.
\end{prop}
\begin{proof}Let $\mu\in \mathcal M_N$ be the additive function the values of which on atoms are given by 
\[
\mu(a_I)=\begin{cases}0&\text{ if } I\neq \emptyset\\1&\text { if }I=\emptyset.\end{cases}
\]
Then $\mu(u_I)=1$, for all $\emptyset\neq I\subseteq[N]$. Applying \eqref{decomp} to $\mu$, we obtain
\[
\sum_{\emptyset\neq I\subseteq [N]}m_{f,I}=\sum_{\emptyset\neq I\subseteq [N]}m_{f,I}\mu(u_I)=\mu(f)=\sum_{a_I\subseteq f}\mu(a_I)=\begin{cases}0&\text{ if } a_{\emptyset}\not\subseteq f\\1&\text { if }a_{\emptyset}\subseteq f.\end{cases}\qedhere
\] 
\end{proof}
\begin{prop}\label{Kronecker}
For any $\emptyset\neq I\subseteq [N]$, $m_{u_I,J}=\delta_{I,J}$ holds, where $\delta_{I,J}$ is the Kronecker delta symbol.
\end{prop}
\begin{proof}
It is clear that $\mu(u_I)=\sum_{\emptyset\neq J\subseteq [N]}\delta_{I,J}\mu(u_J)$. By the uniqueness of the coefficients  $m_{u_I,J}$, this equation implies $m_{u_I,J}=\delta_{I,J}$.
\end{proof}
\begin{prop}\label{m_additivity}
If $f,g\in\mathcal C_N$ and $f\cap g=\emptyset$, then $m_{f\cup g,I}=m_{f,I}+m_{g,I}$ for every $\emptyset\neq I\subseteq [N]$.
\end{prop}
\begin{proof} Since for any $\mu\in \mathcal M_N$, equation 
\[
\sum_{\emptyset\neq I\subseteq [N]}m_{f\cup g,I}\mu(u_I)=\mu(f\cup g)=\mu(f)+\mu(g)=\sum_{\emptyset\neq I\subseteq [N]}(m_{f,I}+ m_{g,I})\mu(u_I)
\]
holds, uniqueness of the coefficients $m_{f\cup g,I}$ implies the statement.
\end{proof}

If $\mu\in\mathcal M_N$, then there are infinitely many ways to extend $\mu$ to a map $\mu\colon \mathcal B_N\to\mathbb R$ preserving the additivity property $\mu(f\cup g)=\mu(f)+\mu(g)-\mu(f\cap g)$. Since such a map is uniquely defined by its values on the atomic expressions  $a_I$, and $\mu(a_I)$ is already given for $I\neq [N]$, the extension of $\mu$ is uniquely given if we prescribe the value $\mu(a_{[N]})\in \mathbb R$. This value is uniquely determined if we require that $\mu(X)=0$, since this equation holds if and only $\mu(a_{[N]})=-\sum_{I\subsetneq [N]}\mu(a_I)$.

\begin{defin} The unique extension of $\mu\in\mathcal M_N$ to a map $\mu\colon \mathcal B_N$ satisfying the conditions $\mu(f\cup g)=\mu(f)+\mu(g)-\mu(f\cap g)$ and $\mu(X)=0$ will be called the \emph{$0$-weight extension of $\mu$}. 
\end{defin}

\section{Asymptotics for the volume of Boolean expressions of large balls\label{sec:4}}

Let $f\in \mathcal C_N$ be a Boolean expression built from the variables $x_1,\dots,x_N$ and the operations $\cup$, $\cap$ and $\setminus$. For a system of $N$ points $\mathbf p=(\mathbf p_1,\dots, \mathbf p_N)\in (\mathbb R^d)^N$ and a given radius $r>0$, consider the body 
\[
B^d_f(\mathbf p,r)=f(B^d(\mathbf p_1,r),\dots,B^d(\mathbf p_N,r))
\]
obtained by evaluating $f$ on the balls $x_i=B^d(\mathbf p_i,r)$. We are interested in the asymptotic behaviour of the volume $\mathcal V^d_f(\mathbf p,r)=\Vol_d(B^d_f(\mathbf p,r))$ of this body.

For a system of points $\mathbf p=(\mathbf p_1,\dots,\mathbf p_N)\in (\mathbb R^d)^N$ and a set $I\subseteq [N]$, denote by $K_I(\mathbf p)$ the convex hull of the points $\{\mathbf p_i\mid i\in I\}$.

\begin{defin}
For $f\in \mathcal C_N$ and a system of points $\mathbf p\in (\mathbb R^d)^N$, define the \emph{Boolean quermassintegrals} $W_{f,k}^d(\mathbf p)$ and \emph{Boolean intrinsic volumes} $V_{f,k}(\mathbf p)$
by the equations 
\[
W_{f,k}^d(\mathbf p)=\sum_{\emptyset\neq I\subseteq [N]}m_{f,I}W_k^d(K_I(\mathbf p))\quad\text{ and }\quad
V_{f,k}(\mathbf p)=\sum_{\emptyset\neq I\subseteq [N]}m_{f,I}V_k(K_I(\mathbf p)).
\]
By Proposition \ref{m_additivity}, for any $k$ and $\mathbf p\in(\mathbb R^d)^N$, the maps $\mathcal C_N\ni f\mapsto W_{f,k}^d(\mathbf p)$ and $\mathcal C_N\ni f\mapsto V_{f,k}(\mathbf p)$ are in $\mathcal M_N$. We define $W_{f,k}^d(\mathbf p)$ and $V_{f,k}(\mathbf p)$ \emph{for arbitrary $f\in \mathcal B_N$} as the $0$-weight extension of these maps, respectively.
\end{defin}

\begin{thm}\label{flower_volume} For any Boolean expression $f\in \mathcal C_N$, and any fixed system of centers $\mathbf p\in (\mathbb R^d)^N$ we have
$$
\mathcal V_f^d(\mathbf p,r)= \sum_{k=d-2}^d\binom{d}{k}W_{f,k}^d(\mathbf p)r^k+O(r^{d-3})=\sum_{k=0}^2 \kappa_{d-k}V_{f,k}(\mathbf p)r^{d-k}+O(r^{d-3}).
$$
\end{thm}
\begin{proof}
Let $\mu\colon\mathcal C_N\to \mathbb R$ be the additive function defined by $\mu(g)=\mathcal V_g^d(\mathbf p,r)$. Applying equation \eqref{decomp} for $\mu$, we obtain
\[
\mathcal V_f^d(\mathbf p,r)=\mu(f)=\sum_{\emptyset\neq I\subseteq [N]}m_{f,I}\mu(u_I)=\sum_{\emptyset\neq I\subseteq [N]}m_{f,I}\mathcal V_{u_I}^d(\mathbf p,r).
\]
For each $I$, $\mathcal V_{u_I}^d$ is the volume of the union of some balls, to which we can apply Corollary \ref{Cor:union_asymptotics}. This gives 
\begin{equation} 
\mathcal V_{u_I}^d=\kappa_dr^d+\kappa_{d-1}V_1(K_I)r^{d-1}+\kappa_{d-2}V_2(K_I)r^{d-2}+O(r^{d-3}).
\end{equation} 
The last two equations together with the definition of the Boolean quermassintegrals and Boolean intrisic volumes imply the theorem.
\end{proof}
\begin{rem} One of the main goals set in the introduction was to extend equations \eqref{union_asymptotics}, \eqref{intersection_asymptotics}, and \eqref{eq:union} for the volumes of Boolean expressions of large congruent balls, finding suitable generalizations of the intrinsic volumes $V_0$, $V_1$, $V_2$, appearing in \eqref{eq:union}. Theorem \ref{flower_volume} gives the desired extension and justifies our definition of the Boolean intrinsic volumes. 
\end{rem}
\section{Properties of Boolean intrinsic volumes\label{sec:5}}

The following properties are straightforward corollaries of the analogous properties of intrinsic volumes of convex bodies and the definitions.
\begin{prop}  \label{basic_properties}\mbox{}
\begin{itemize}
\item[(a)] $V_{f,0}(\mathbf p)$ does not depend on $\mathbf p$. Its value  $V_{f,0}\equiv \sum_{\emptyset\neq I\subset [N]}m_{f,I}$ is $1$ if $a_{\emptyset}\subseteq f$, and $0$ otherwise.
\item[(b)] The Boolean intrinsic volume $V_{f,k}(\mathbf p)$ does not depend on the dimension $d$.  In particular,
\[
W_{f,k}^d=\frac{\kappa_{k}}{\binom{d}{k}}V_{f,d-k}=\frac{\kappa_{k}}{\binom{d}{k}}V_{f,(d+s)-(k+s)}=\frac{\binom{d+s}{k+s}\kappa_{k}}{\binom{d}{k}\kappa_{k+s}}W^{d+s}_{f,k+s}=\frac{(d+1)\cdots(d+s)\kappa_{k}}{(k+1)\cdots(k+s)\kappa_{k+s}}W^{d+s}_{f,k+s}
\]
for any $s\in \mathbb N$.
\item[(c)] $V_{f,k}$ is a continuous function on $(\mathbb R^d)^N$ for every $d>0$.
\item[(d)] If $f,g\in \mathcal B_N$ and $f\cap g=\emptyset$, then $W_{f\cup g,k}^d=W_{f,k}^d+W_{g,k}^d$ and $V_{f\cup g,k}=V_{f,k}+V_{g,k}$.
\item[(e)] $W_{\bar f,k}^d=-W_{f,k}^d$ and $V_{\bar f,k}=-V_{f,k}$ for any $f\in \mathcal B_N$.
\end{itemize}
\end{prop}

We are going to find a formula for the Boolean intrinsic volumes that generalizes equation \eqref{intrinsic_volume}. 
Assume that any $k+2$ points of the system $\mathbf p=(\mathbf p_1,\dots,\mathbf p_N)\in (\mathbb R^d)^N$ are affinely independent. This can always be achieved by a small perturbation of the points if $d\geq k+1$. Choose a $k+1$ element index set   $S=\{{i_1},\dots,{i_{k+1}}\}\subset [N]$ and denote by $\sigma_S$ the convex hull of the points $\mathbf p_{i_1},\dots,\mathbf p_{i_{k+1}}$. By the general position assumption on $\mathbf p$, $\sigma_S$ is a $k$-dimensional simplex and the affine subspace $\langle \sigma_S\rangle$ spanned by it does not contain any of the points $\mathbf p_j$ for $j\notin S$. 

Define an integer valued function $n_{f,S,\mathbf p}\colon \mathbb S^{d-k-1}_S\to \mathbb Z$ on the unit sphere $\mathbb S^{d-k-1}_S=\{\mathbf u\in \mathbb S^{d-1}\mid \mathbf u\perp \langle \sigma_S\rangle \}$ as follows. Choose a vector $\mathbf u\in \mathbb S^{d-k-1}_S$. Split the 
index set $[N]$ into three parts depending on the position of the point $\mathbf p_i$ relative to the hyperplane orthogonal to $\mathbf u$, containing the simplex $\sigma_S$ by setting
\begin{align*}
\Pi_+&=\{j\in [N]\mid \langle \mathbf p_j-\mathbf p_{i_1},\mathbf u\rangle >0\},\\
\Pi_0\,&=\{j\in [N]\mid \langle \mathbf p_j-\mathbf p_{i_1},\mathbf u\rangle =0\},\\
\Pi_-&=\{j\in [N]\mid \langle \mathbf p_j-\mathbf p_{i_1},\mathbf u\rangle <0\}.
\end{align*} 
It is clear that $S\subseteq \Pi_0$ and $S=\Pi_0$ for almost all $\mathbf u$. Define the elements $y_1,\dots,y_N\in \mathcal B_N$ by the rule
\[
y_j=\begin{cases}
X&\text{ if }j\in \Pi_+\cup(\Pi_0\setminus S),\\
x_j&\text{ if }j\in S,\\
\emptyset&\text{ if }j\in \Pi_-.
\end{cases}
\]
Evaluating the Boolean expression $f$ on the $y_j$'s we obtain an element $f(y_1,\dots,y_N)\in \mathcal B_{k+1}(x_{i_1},\dots,x_{i_{k+1}})$ in the free Boolean algebra generated by the elements $x_{i_1},\dots,x_{i_{k+1}}$.
Set $n_{f,S,\mathbf p}(\mathbf u)= (-1)^{k+1}\tilde\chi_{k+1}(f(y_1,\dots,y_N))$.

The values of $n_{f,S,\mathbf p}$ are integers in the interval $[-2^k,2^k]$. Let 
$$
\nu_{f,S,\mathbf p}= \frac{1}{(d-k)\kappa_{d-k}} \int_{\mathbb S^{d-k-1}_S}n_{f,S,\mathbf p}(\mathbf u)\mathrm{d}\mathbf u
$$
be the average value of $n_{f,S,\mathbf p}$.

\begin{thm}\label{thm:2}
If $f\in \mathcal B_N$ and $\mathbf p\in (\mathbb R^d)^N$ satisfies that any $k+2$ points of $\mathbf p$ are affinely independent, then  we have
\begin{equation}\label{valyu}
V_{f,k}(\mathbf p)=\sum_{\genfrac{}{}{0pt}{}{S\subseteq [N]}{|S|=k+1}} \nu_{f,S,\mathbf p}\Vol_k(\sigma_S).
\end{equation} 
\end{thm}
\begin{proof}
If $f,g\in \mathcal B_N$ are disjoint, that is $f\cap g=\emptyset$, then $V_{f\cup g,k}=V_{f,k}+V_{g,k}$, furthermore, $f(y_1,\dots,y_N)\cap g(y_1,\dots,y_N)=\emptyset$ for any choice of the variables $y_i$,  and since the reduced Euler characteristic is an additive function, $\nu_{f\cup g,k} =\nu_{f,k}+\nu_{g,k}$. Thus, both sides of equation \eqref{valyu} are additive functions of the Boolean expression $f$. Since both sides vanish for $f=X$, the two sides are equal for any $f\in \mathcal B_N$ if they are equal for any $f\in\mathcal C_N$. As it was shown in the proof of Proposition \ref{prop:decomp}, the evaluations $\ev_{u_I}$, for $\emptyset\neq I\subseteq [N]$, form a basis of $\mathcal M_N^*$, so it is enough to check the proposition for the unions $u_I$.

Assume $f=u_I$. Then $V_{f,k}(\mathbf p)=V_k(K_I(\mathbf p))$ by Proposition \ref{Kronecker}. Let $S=\{i_1,\dots,i_{k+1}\}\subseteq [N]$ be a set of $k+1$ indices. To understand the geometrical meaning of $n_{f,S,\mathbf p}(\mathbf u)$, consider first the value of $f(y_1,\dots,y_N)=\bigcup_{j\in I} y_j$. 

If $y_j=X$ for an index $j\in I$, then $f(y_1,\dots,y_N)=X$ and $\nu_{f,S}(\mathbf u)=\tilde\chi_{k+1}(X)=0$. Hence $n_{f,S,\mathbf p}(\mathbf u)$ vanishes if $I\not \subseteq \Pi_-\cup \Pi_0$.   By Proposition \ref{zero_Euler}, $n_{f,S,\mathbf p}(\mathbf u)$ vanishes also in the case when one of the variables $x_{i_1},\dots,x_{i_{k+1}}$ does not appear in $f(y_1,\dots,y_N)$. These variables appear in $f(y_1,\dots,y_N)$ if and only if $S\subseteq I\cap \Pi_0$. This means that if $n_{f,S,\mathbf p}(\mathbf u)\neq 0$, then $K_I(\mathbf p)$ is contained in the halfspace $\{\mathbf x\in \mathbb R^d\mid \langle \mathbf u,\mathbf x-\mathbf p_{i_1}\rangle\leq 0 \}$ and the boundary hyperplane of this halfspace intersects the polytope $K_I(\mathbf p)$ in a face that contains the $k$-dimensional simplex $\sigma_S$. What is the value of $n_{f,S,\mathbf p}(\mathbf u)$ in this case? If $I\subseteq \Pi_-\cup \Pi_0$ and $S\subseteq I\cap \Pi_0$, then 
\[
n_{f,S,\mathbf p}(\mathbf u)=(-1)^{k+1}\tilde\chi_{k+1}(x_{i_1}\cup\dots\cup x_{i_{k+1}})=-\tilde\chi_{k+1}(x_{i_1}\cap\dots\cap x_{i_{k+1}})=1.
\]
If the simplex $\sigma_S$ is not a face of $K_I(\mathbf p)$, then the smallest face of of $K_I(\mathbf p)$ that contains $\sigma_S$ has dimension bigger than $k$ because of the general position assumption on $\mathbf p$. In this case, the support of the function $n_{f,S,\mathbf p}$ is contained in a great subsphere of $\mathbb S_S^{d-k-1}$, and $\nu_{f,S,\mathbf p}=0$.

If $\sigma_S$ is a face of $K_I(\mathbf p)$, then $n_{f,S,\mathbf p}$ is the indicator function of the intersection of the cone $N(\sigma_S,K_I(\mathbf p))$ and the sphere $\mathbb S_S^{d-k-1}$, therefore 
\[
\nu_{f,S,\mathbf p}=\frac{1}{(d-k)\kappa_{d-k}} \int_{\mathbb S^{d-k-1}_S}n_{f,S,\mathbf p}(\mathbf u)\mathrm{d}\mathbf u=\frac{\Vol_{d-k}(n(\sigma_S,K_I(\mathbf p)))}{\kappa_{d-k}}=\nu(\sigma_S, K_I(\mathbf p)).
\]
As all the $k$-dimensional faces of $K_I(\mathbf p))$ are simplicies, we conclude that for $f=u_I$, we have
\[
\sum_{\genfrac{}{}{0pt}{}{S\subseteq [N]}{|S|=k+1}} \nu_{f,S,\mathbf p}\Vol_k(\sigma_S)=\sum_{\sigma\in \mathcal F_k(K_I(\mathbf p))}\nu(\sigma,K_I(\mathbf p))\Vol_k(L)=V_{k}(K_I(\mathbf p))=V_{f,k}(\mathbf p),
\]
as desired.
\end{proof}
\begin{prop}\label{prop:10} If $f \in \mathcal B_N$, $f^*$ and $f^{\bar *}$ are the dual and contradual of $f$ respectively, then
\[
V_{f^*,k}=-V_{f^{\bar *},k}=(-1)^kV_{f,k}\text{ and }W_{f^*,k}^d=-W_{f^{\bar *},k}^d=(-1)^{d-k}W_{f,k}^d.
\]
\end{prop}
\begin{proof} Due to Proposition \ref{basic_properties} (e) and (b), it is enough to show the equality $V_{f^*,k}=(-1)^kV_{f,k}$. As $V_{f,k}(\mathbf p)$ does not depend on the dimension of the ambient space $\mathbb R^d$, we may assume that $d>k$. Then the set of configurations $\mathbf p=(\mathbf p_1,\dots,\mathbf p_N)\in(\mathbb R^d)^N$ satisfying that any $k+2$ of the points $\mathbf p_1,\dots,\mathbf p_N$ are affinely independent is dense in $(\mathbb R^d)^N$. Since $V_{f,k}$ is continuous on $(\mathbb R^d)^N$ for all $f\in \mathcal B_N$, it suffices to prove the equation $V_{f^*,k}(\mathbf p)=(-1)^kV_{f,k}(\mathbf p)$ for configurations satisfying this general position condition. Under this assumption, Theorem \ref{thm:2} implies the statement if we show the equations $\nu_{f^*,S,\mathbf p}=(-1)^k\nu_{f,S,\mathbf p}$. 

Consider the function $n_{f,S,\mathbf p}(\mathbf u)=(-1)^{k+1}\tilde\chi_{k+1}(f(y_1,\dots,y_N))$ in the definition of $\nu_{f,S,\mathbf p}$. It is not difficult to see that $f^{\bar *}(y_1,\dots,y_N)$ is the contradual of $f(y_1,\dots,y_N)$ and $f^{*}(y_1,\dots,y_N)$ is the dual of it, so applying Proposition \ref{Euler_dual}, we obtain
$n_{f^*,S,\mathbf p}=(-1)^k n_{f,S,\mathbf p}$. Taking the mean value of both sides over the unit sphere $\mathbb S_S^{d-k-1}$ we get the desired equation $\nu_{f^*,S,\mathbf p}=(-1)^k\nu_{f,S,\mathbf p}$.
\end{proof}

Denote by $l_{\mathbf u}\colon \mathbb R^d\to \mathbb R$
the linear function $l_{\mathbf u}\colon \mathbf x\mapsto \langle \mathbf u,\mathbf x\rangle$. If $\mathbf u$ is a unit vector, and $K$ is a bounded convex set, then the length of the interval $l_{\mathbf u}(K)$ is the width $w_K(\mathbf u)$ of $K$ in the direction of $\mathbf u$. It is known that $V_1(K)$ is proportional to the mean width of $K$, namely,
\[
V_1(K)=\frac{1}{2\kappa_{d-1}}\int_{\mathbb S^{d-1}}w_K(\mathbf u)\mathrm d \mathbf u
=\frac{d\kappa_d}{2\kappa_{d-1}}\boldsymbol{\omega}_d(K).
\]
The width and the mean width can be expressed with the help of the support function of $K$. Recall that the support function of a bounded set $X\subset \mathbb R^d$ is defined as the function $h_X\colon \mathbb S^{d-1}\to\mathbb R$, $h_X(\mathbf u)=\sup_{\mathbf x\in X}\langle \mathbf x,\mathbf u\rangle$. It is clear that $w_K(\mathbf u)=h_K(\mathbf u)+h_K(-\mathbf u)$, and 
\[
V_1(K)=\frac{1}{2\kappa_{d-1}}\int_{\mathbb S^{d-1}}(h_K(\mathbf u)+h_K(-\mathbf u))\mathrm d \mathbf u=\frac{1}{\kappa_{d-1}}\int_{\mathbb S^{d-1}}h_K(\mathbf u)\mathrm d \mathbf u.
\]
We can extend this formula for the case when $f\in\mathcal L_N$. Then $f$ can be evaluated on real numbers by setting $a\cup b=\max\{a,b\}$ and $a\cap b=\min\{a,b\}$ for $a,b\in\mathbb R$.
\begin{thm}\label{thm:3} If $f\in \mathcal L_N$, then for any $\mathbf p\in (\mathbb R^d)^N$, we have
\[
V_{f,1}(\mathbf p)=\frac{1}{\kappa_{d-1}}\int_{\mathbb S^{d-1}}f(\langle \mathbf u,\mathbf p_1\rangle,\dots,\langle \mathbf u,\mathbf p_N\rangle)\mathrm{d}\mathbf u.
\]
\end{thm}
\begin{proof} Suppose that the points $\mathbf p_i$ are all contained in the interior of the ball $B_R=B^d(\mathbf 0,R)$. Since $f\in \mathcal L_N$, $a_\emptyset\subseteq f$, therefore $\sum_{\emptyset\neq I\subseteq [N] }m_{f,I}=1$, and    
\[
V_{f,1}(\mathbf p)=\sum_{\emptyset\neq I\subseteq [N]} m_{f,I}V_1(K_I(\mathbf p))=\left(\sum_{\emptyset\neq I\subseteq [N]} m_{f,I}V_1(K_I(\mathbf p)+B_R)\right)-V_1(B_R).
\]
Denote by $S_i(\mathbf u)$ the interval $l_{\mathbf u}(\{\mathbf p_i\}+B_R)=[\langle\mathbf u,\mathbf p_i\rangle -R,\langle\mathbf u,\mathbf p_i\rangle +R]$. Then
\[
V_1(K_I(\mathbf p)+B_R)=\frac{1}{2\kappa_{d-1}}\int_{\mathbb S^{d-1}}\Vol_1(l_{\mathbf u} (K_I+B_R)\mathrm d \mathbf u=\frac{1}{2\kappa_{d-1}}\int_{\mathbb S^{d-1}}\Vol_1\left(\bigcup_{i\in I}S_i(\mathbf u) \right)\mathrm d \mathbf u,
\]
and
\[
V_{f,1}(\mathbf p)=\frac{1}{2\kappa_{d-1}}\int_{\mathbb S^{d-1}}\Big(\sum_{\emptyset\neq I\subseteq [N]} m_{f,I}\Vol_1\Big(\bigcup_{i\in I}S_i(\mathbf u) \Big)\Big)\mathrm d \mathbf u-\frac{d\kappa_d R}{\kappa_{d-1}}.
\]
For any fixed $\mathbf u\in \mathbb S^{d-1}$, the function $\mu\colon \mathcal C_N\to \mathbb R$ defined by $\mu(f)=\Vol_1(f(S_1(\mathbf u),\dots,S_N(\mathbf u)))$ is in $\mathcal M_N$, therefore Proposition \ref{prop:decomp} yields
\[
\sum_{\emptyset\neq I\subseteq [N]} m_{f,I}\Vol_1\Big(\bigcup_{i\in I}S_i(\mathbf u) \Big)=\sum_{\emptyset\neq I\subseteq [N]} m_{f,I}\mu(u_I)=\mu(f)= \Vol_1(f(S_1(\mathbf u),\dots,S_N(\mathbf u)).
\]
By the choice of $R$, $0$ is a common interior point of all the intervals $S_i(\mathbf u)$. For this reason, all the sets that can be obtained from these intervals using the operations $\cup$ and $\cap$ are also intervals. In particular,
\[
f(S_1(\mathbf u),\dots,S_N(\mathbf u))=[-f(-\langle \mathbf u,\mathbf p_1\rangle,\dots,-\langle \mathbf u,\mathbf p_N\rangle)-R,f(\langle \mathbf u,\mathbf p_1\rangle,\dots,\langle \mathbf u,\mathbf p_N\rangle)+R],
\]
and 
\[
 \Vol_1(f(S_1(\mathbf u),\dots,S_N(\mathbf u))=f(\langle \mathbf u,\mathbf p_1\rangle,\dots,\langle \mathbf u,\mathbf p_N\rangle)+f(\langle -\mathbf u,\mathbf p_1\rangle,\dots,\langle -\mathbf u,\mathbf p_N\rangle)+2R.
\]
Using the fact that for any integrable function $h\colon \mathbb S^{d-1}\to \mathbb R$, we have $\int_{\mathbb S^{d-1}}h(\mathbf u)\mathrm{d}\mathbf u=\int_{\mathbb S^{d-1}}h(-\mathbf u)\mathrm{d}\mathbf u$, these equations give
\begin{align*}
V_{f,1}(\mathbf p)&=\frac{1}{2\kappa_{d-1}}\int_{\mathbb S^{d-1}}(f(\langle \mathbf u,\mathbf p_1\rangle,\dots,\langle \mathbf u,\mathbf p_N\rangle)+f(\langle -\mathbf u,\mathbf p_1\rangle,\dots,\langle -\mathbf u,\mathbf p_N\rangle)+2R)\mathrm{d}\mathbf{u}-\frac{d\kappa_d R}{\kappa_{d-1}}\\
&=\frac{1}{\kappa_{d-1}}\int_{\mathbb S^{d-1}}f(\langle \mathbf u,\mathbf p_1\rangle,\dots,\langle \mathbf u,\mathbf p_N\rangle)\mathrm{d}\mathbf{u},
\end{align*}
as we wanted to show.
\end{proof}

\section{Monotonocity of the Boolean intrinsic volume \texorpdfstring{$V_{f,1}$}{V(f,1)}\label{sec:6}}

In this section, we prove the following result.
\begin{thm}\label{T:V_f,1} Assume that the Boolean expression $f\in \mathcal C_N$ can be represented by a formula in which  each of the variables occurs exactly once. Define the signs $\epsilon_{ij}^f$, for  $1\leq i<j\leq N$, as in the introduction.  If the configurations $\mathbf p=(\mathbf 
p_1,\dots,\mathbf p_N)$ and $\mathbf q=(\mathbf q_1,\dots,\mathbf q_N)\in (\mathbb R^d)^N$ 
satisfy the inequalities $\epsilon_{ij}^f(\d(\mathbf p_i,\mathbf p_j)- 
\d(\mathbf q_i,\mathbf q_j))\geq 0$ for all $0\leq i<j\leq N$, then we have
\begin{equation}\label{V_f,1_ineq}
V_{f,1}(\mathbf p)\geq V_{f,1}(\mathbf q).
\end{equation}
\end{thm}
\begin{proof} 
It is proved in \cite{Csikos_Schlafli}, that if there exist piecewise analytic continuous maps $\mathbf z_i\colon[0,1]\to \mathbb R^d$ for $1\leq i\leq N$, such that $\mathbf z_i(0)=\mathbf p_i$, $\mathbf z_i(1)=\mathbf q_i$, and the distances $\d(\mathbf z_i(t),\mathbf z_j(t))$ are weakly monotonous functions of $t$ for all $i$ and $j$, then ineqality  \eqref{general_Kneser} is true for any choice of the radii. It is not difficult to see that the analytic curves $\mathbf z_i\colon[0,1]\to \mathbb R^d\times \mathbb R^d$ defined by $\mathbf z_i(t)=(\cos(t\pi/2)\mathbf p_i,\sin(t\pi/2)\mathbf q_i)$ connect the points $(\mathbf p_i,\mathbf 0)$ to the points $(\mathbf 0,\mathbf q_i)$ in the required way, but jumping into $\mathbb R^{2d}$. Thus, embedding the centers into $\mathbb R^{2d}$, our assumptions imply the inequality
\begin{equation}\label{general_Kneser_2d}
\mathcal V_f^{2d}(\mathbf p,r)=\vol_{2d}\left(B^{2d}_j(\mathbf p,r)\right)\geq 
\vol_{2d}\left(f( B^{2d}_f(\mathbf q,r)\right)=\mathcal V_f^{2d}(\mathbf q,r)
\end{equation}
for any choice of the radius $r$. By Proposition \ref{basic_properties} (a), $V_{f,0}(\mathbf p)=V_{f,0}(\mathbf q)$, therefore Theorem \ref{flower_volume} gives
\[
0\leq \mathcal V^{2d}_f(\mathbf p,r)- \mathcal V^{2d}_f(\mathbf q,r)=\kappa_{2d-1}(V_{f,1}(\mathbf p)-V_{f,1}(\mathbf q))r^{2d-1}+O(r^{2d-2}).
\]
This inequality can hold for large $r$ only if the coefficient of the dominant term is nonnegative, i.e.,  $V_{f,1}(\mathbf p)\geq V_{f,1}(\mathbf q)$.
\end{proof}

It seems to be an interesting question whether we can write strict inequality in \eqref{V_f,1_ineq} if, in addition to the assumptions of Theorem \ref{T:V_f,1}, we know that the configurations  $\mathbf p$ and $\mathbf q$ are not congruent. An affirmative answer would imply that the generalized Kneser--Poulsen conjecture holds for Boolean expression of congruent balls if the radius of the balls is greater than a certain number depending on the system of the centers.

\section{Acknowledgements}
This research was supported by the Hungarian National Science and Research 
Foundation OTKA K 112703. Part of the research was done in the academic year 2014/15, while the author enjoyed the 
hospitality of the MTA Alfr\'ed R\'enyi Institute of Mathematics as a guest 
researcher. 		
		
\bibliographystyle{spsmci}
\bibliography{flower_weight}
\end{document}